\documentclass[11pt,reqno]{amsart}
\usepackage{amssymb,latexsym}
\usepackage{fancyhdr,amssymb}
\usepackage{hyperref}
\usepackage{listings}
\usepackage{amsmath}
\usepackage{hyperref}
\usepackage[utf8]{inputenc}
\usepackage{enumerate}
\usepackage{ucs}
\usepackage{lmodern}
\usepackage[T1]{fontenc}
\usepackage{mathrsfs}

\usepackage[letterpaper, top = 4cm, bottom = 4cm,right = 4cm, left = 4cm]{geometry}
\theoremstyle{theorem}

\usepackage{comment}
\newtheorem{theorem}{Theorem}[section]
\newtheorem{lemma}[theorem]{Lemma}
\newtheorem{proposition}[theorem]{Proposition}
\newtheorem{cor}[theorem]{Corollary}

\theoremstyle{definition}
\newtheorem{definition}[theorem]{Definition}

\theoremstyle{remark}
\newtheorem{remark}{Remark}

\theoremstyle{example}

\title{Progress towards a nonintegrality conjecture}
\author[S. Laishram, D. L\'opez-Aguayo, C. Pomerance, and T. Thongjunthug]{Shanta Laishram, Daniel L\'{o}pez-Aguayo, Carl Pomerance and Thotsaphon Thongjunthug}

\address{Stat-Math Unit, India Statistical Institute, 7, S.J.S Sansanwal Marg, New Delhi, 110016, India.}
\email{shanta@isid.ac.in}
\address{Tecnologico de Monterrey, Escuela de Ingenier\'{i}a y Ciencias, Monterrey, Nuevo Le\'{o}n, M\'{e}xico.}
\email{dlopez.aguayo@tec.mx}
\address{Department of Mathematics, Dartmouth College, Hanover, NH 03755, USA.}
\email{carlp@math.dartmouth.edu}
\address{Department of Mathematics, Faculty of Science, Khon Kaen University, Khon Kaen 40002, Thailand.}
\email{thotho@kku.ac.th}

\begin{document}
\maketitle

\begin{abstract} 
Given $r \in \mathbb{N}$, define the function $S_{r}: \mathbb{N} \rightarrow \mathbb{Q}$ by
\begin{center}
$S_{r}(n)=\displaystyle \sum_{k=0}^{n} \frac{k}{k+r} \binom{n}{k}$.
\end{center}
In $2015$, the second author conjectured that there are infinitely many $r \in \mathbb{N}$ such that $S_{r}(n)$ is nonintegral for all $n \geq 1$, and proved that $S_{r}(n)$ is not an integer for $r \in \{2,3,4\}$ and for all $n \geq 1$. In $2016$, Florian Luca and the second author raised the stronger conjecture that for any $r \geq 1$, $S_{r}(n)$ is nonintegral for all $n \geq 1$. They proved that $S_{r}(n)$ is nonintegral for $r \in \{5,6\}$ and that $S_{r}(n)$ is not an integer for any $r  \geq 2$ and $1 \leq n \leq r-1$. In particular, for all $r \geq 2$, $S_{r}(n)$ is nonintegral for at least $r-1$ values of $n$. In $2018$, the fourth author gave sufficient conditions for the nonintegrality of $S_{r}(n)$ for all $n \geq 1$, and derived an algorithm to sometimes determine such nonintegrality; along the way he proved that $S_{r}(n)$ is nonintegral for $r \in \{7,8,9,10\}$ and for all $n \geq 1$. By improving this algorithm we prove the conjecture for $r\le 22$. Our principal result is that $S_r(n)$ is usually nonintegral in that  the upper asymptotic density of the set of integers $n$ with $S_r(n)$ integral decays faster than
any fixed power of $r^{-1}$ as $r$ grows.  
\end{abstract}

\section{Introduction}
In 2014, Marcel Chiri\unichar{539}\unichar{259} \cite{3} asked to show that $\displaystyle \sum_{k=0}^{n} \frac{k}{k+1} \binom{n}{k}$ is nonintegral for all integers $n \geq 1$. This is true and one can prove it as follows: the given sum is equal to $2^{n}-\frac{2^{n+1}-1}{n+1}$ and $\frac{2^{n+1}-1}{n+1}$ is never an integer due to the fact that for every integer $r \geq 2$, $2^{r} \not \equiv 1 \pmod{r}$. Given $r \in \mathbb{N}$, define the function $S_{r}: \mathbb{N} \rightarrow \mathbb{Q}$ by $S_{r}(n)=\displaystyle \sum_{k=0}^{n} \frac{k}{k+r} \binom{n}{k}$. Motivated by \cite{3}, the second author \cite{7} raised the question whether there are infinitely many $r \in \mathbb{N}$ such that $S_{r}(n)$ is nonintegral for all $n \geq 1$, and proved that $S_{r}(n)$ is not an integer for $r \in \{2,3,4\}$ and for all $n \geq 1$. These results also hinge on the fact that for $r \geq 2$, $2^{r} \not \equiv 1 \pmod{r}$.

In 2016, Florian Luca and the second author \cite{8} conjectured that for any $r \geq 2$, $S_r(n)$ is always nonintegral for all $n \geq 1$.
They proved that $S_{r}(n)$ is nonintegral for $r \in \{5,6\}$ and that $S_{r}(n)$ is not an integer for any $r  \geq 2$ and $1 \leq n \leq r-1$. In particular, for all $r \geq 2$, $S_{r}(n)$ is nonintegral for at least $r-1$ values of $n$. The proof of this fact relies heavily on the following theorem of Sylvester \cite{4}: every product of $k$ consecutive integers larger than $k$ is divisible by a prime larger than $k$.  In $2018$, the fourth author \cite{10} gave sufficient conditions for the nonintegrality of $S_{r}(n)$ for all $n \geq 1$, and derived an algorithm to sometimes determine such nonintegrality; along the way he proved that $S_{r}(n)$ is nonintegral for $r \in \{7,8,9,10\}$ and for all $n \geq 1$. 

Our first result is a simplified version of \cite[Theorem 4.1]{10}.
\begin{theorem}
\label{thm:tt}
Given an integer $r\ge2$, the sum $S_{r}(n)$ is nonintegral for all $n\in\mathbb N$ if the following condition holds.
With $m_r$ the product of all primes up to $r$,
each integer $n \in \{1,\ldots,m_r\}$ satisfies at least one of the following: 
\begin{enumerate}
\item[\rm(1)] There exists $i \in \{1, \ldots, r\}$ such that $\gcd(n+i,m_r)=1$. 
\item[\rm(2)] There exist $i,j \in \{1,\ldots,r\}$ such that $\gcd(n+i,m_r)=\gcd(n+j,m_r)=2$ and
$0<|i-j|<8$.
\end{enumerate}
\end{theorem}
We discuss the more complicated version of this result from \cite{10} in Section \ref{sec:tt}.
With our simpler criterion, plus some other ideas presented below, we are able to
prove the following result.
\begin{theorem}\label{thm:r16}
For $1\le r\le 22$ we have $S_r(n)$ nonintegral for all $n\in\mathbb{N}$.
\end{theorem}
We also prove the following result.
\begin{theorem}
\label{thm:dense}
For each $k>0$ there exists a constant $c_{k}>0$, such that for each integer $r>1$, the upper asymptotic density of $\{n\in\mathbb{N}:S_r(n)\in\mathbb{N}\}$ is at most $c_{k}/r^{k}$.
\end{theorem}
The proof uses a theorem by Montgomery-Vaughan \cite{9}. 

This paper is organized as follows. In Section \ref{sec2} we give the necessary preliminaries. In Section \ref{sec3}
we prove Theorems \ref{thm:tt} and \ref{thm:r16}, in Section \ref{sec:main} we prove \ref{thm:dense}, and in
Section \ref{sec:tt} we discuss the original version of Theorem \ref{thm:tt} from \cite{10}.

\section{Preliminaries} \label{sec2}
Throughout this section, for a prime $p$ and an integer $u$ with $p \nmid u$, we shall let $\operatorname{ord}_{p}u$ denote the least positive integer $k$ such that $u^{k} \equiv 1 \pmod{p}$. 

\begin{lemma}\label{lem:fermat}
Suppose that $p$ is an odd prime dividing $n\in\mathbb{N}$ and $p\mid 2^n-1$.  If $a$ is the largest divisor of $n$ composed of primes smaller than $p$, we have $p\mid 2^a-1$.
\end{lemma}
\begin{proof}
Since $\operatorname{ord}_p2\mid p-1$, all of the primes dividing $\operatorname{ord}_p2$ are smaller
than $p$.  If $p\mid 2^{n}-1$, then $\operatorname{ord}_p2\mid n$, so that $\operatorname{ord}_p2\mid a$.
\end{proof}

Let $\phi$ denote Euler's function from elementary number theory.

\begin{definition}
Let $q\in\mathbb{N}$ and let $\alpha \geq 1$ be a real number. Following \cite{9}, we define
\begin{center}
$V_{\alpha}(q)=\displaystyle \sum_{i=1}^{\phi(q)} (a_{i+1}-a_{i})^{\alpha}$
\end{center}
where $1=a_{1}<a_{2} \cdots < a_{\phi(q)+1}$ are the integers in $[1,q+1]$ that are coprime to $q$. 
\end{definition}
The following proposition is crucial for the proof of our main theorem.
\begin{proposition}{\rm(}\cite[Corollary 1]{9}{\rm)} \label{mont} Let $q\in\mathbb{N}$. For any fixed real number $\alpha \geq 1$, there is a positive number $c(\alpha)$ such that
\begin{center}
$V_{\alpha}(q) \le c(\alpha)\phi(q)(\phi(q)/q)^{-\alpha}$.
\end{center}
\end{proposition}

We have the following
inequality, which follows from \cite[(3.30)]{RS} and a short calculation:
\begin{equation}
    \label{eq:euler}
    \frac{n!}{\phi(n!)}<3\ln n,~~\hbox{ for all }n\ge2.
\end{equation}

Following \cite{8}, we define 
$$
S(r,n):=\sum_{k=0}^{n}\frac r{k+r}\binom nk.$$
It is clear that
\begin{equation}
    \label{eq:2pow}
S_r(n)+S(r,n)=\sum_{k=0}^n\binom nk=2^n,
\end{equation}
so that $S_r(n)$ is integral if and only if $S(r,n)$ is integral.

The following result is shown in \cite{8}.

\begin{lemma}  \label{lem29} 
We have
$$S(r,n)=\displaystyle \sum_{j=1}^{r} (-1)^{r-j} r \binom{r-1}{j-1} \frac{2^{n+j}-1}{n+j}.$$
\end{lemma}

\section{The search to $r=22$} \label{sec3}

\begin{proposition}\label{prop:suff}
Let $n,r\in\mathbb{N}$ and suppose for some integer $j\in\{1,\dots,r\}$ we have $n+j=ab$ where
$a,b\in\mathbb{N}$, $b>1$, such that
\begin{itemize}
    \item each prime dividing $b$ is greater than $r$,
    \item and each prime dividing $2^a-1$ is at most $r$.
\end{itemize}    
Then $S(r,n)$ is nonintegral.
\end{proposition}
\begin{proof}
Let $p$ be the least prime factor of $b$, so that $p>r$.  Note that $p\mid n+j$, but $p$ does not divide any other member of
$\{n+1,\dots,n+r\}$.  Suppose that $p\mid 2^{n+j}-1$.  By Lemma \ref{lem:fermat}, $p\mid 2^a-1$.  But by
assumption, all prime factors of $2^a-1$ are at most $r$, a contradiction.  Thus, in lowest terms, the fraction $(2^{n+j}-1)/(n+j)$ has at
least one factor $p$ in the denominator.  The term corresponding to $j$ in Lemma \ref{lem29} is, up to sign,
$$r\binom{r-1}{j-1}\frac{2^{n+j}-1}{n+j}.$$
So, since $p>r$, we see that this term, when reduced to its lowest terms, has at least one factor $p$ in the denominator.
However, no other term in the sum in Lemma~\ref{lem29} has a factor $p$ in the denominator,
so that in the full sum $S(r,n)$, there is a factor $p$ in the denominator.  That is, $S(r,n)$
is not an integer.  This completes the proof.
\end{proof}
\begin{remark}
We note that for every integer $a\ge2$, there is a prime $p\equiv1\pmod a$ that divides $2^a-1$,
see \cite{B}.  This
implies that with $a,b$ as in Proposition \ref{prop:suff} we have $a\le r-1$, so that $a$ and $b$ are coprime.
\end{remark}

 \begin{proof}[Proof of Theorem \ref{thm:tt}]
 If $\{n+1,\dots,n+r\}$ contains some
 $n+i$ coprime to $m_r$, then we can apply Proposition \ref{prop:suff} to $n+i$ with 
$b=n+i$ and $a=1$.  On the other hand, if $\{n+1,\dots,n+r\}$ contains two even numbers less than 8 apart
 that are not divisible by any odd prime up to $r$,
 at least one of them is not divisible by 8, say it is $n+j$.  We apply Proposition \ref{prop:suff} 
 with $a=2$ or 4 and $b=(n+j)/a$.
   Since we may assume that $n\ge r>6$ from 
 \cite{8}, Proposition \ref{prop:suff} applies.  
 \end{proof}
Using Theorem \ref{thm:tt} and with some additional help from Proposition \ref{prop:suff}
 we can prove the conjecture for $r\le 22$.
 \begin{proof}[Proof of Theorem \ref{thm:r16}]
 We first handle the cases $r=11,12$.  The product of the primes to 11 is $m_{11}=2310$, so it suffices to show that
 every interval of 11 consecutive integers either contains a member coprime to 2310 or contains two even
 members less than 8 apart that are coprime to 1155.  Since the problem is symmetric about 1155, we only
 need to search to this level.  There are precisely 7 intervals of 11 consecutive integers in this range
 which do not contain a number coprime to 2310; these are the intervals starting at 
 \[
 2,114,115,116,200,468,510.\]
 We check that in each of them there are two even numbers less than 8 apart which are coprime to 1155.
 
 The calculation for $r=13,14,15,16$ is somewhat more extensive.  Here we show that every interval of
 13 consecutive integers contains either one that is coprime to $m_{13}=30030$ or contains two even members
 less than 8 apart that are coprime to 15015.  We need only check intervals whose first element is in
 $[1,15015]$.  All but 76 of them have a member coprime to 30030.  Each of these 76 intervals contains
 two even members less than 8 apart that are coprime to 15015.  
 
 This plan breaks down for $r=17,18$.  For example, when $n=60462$, the interval $[n,n+16]$ has each
 member with a nontrivial gcd with $m_{17}=510510$, so that condition (1)
 does not apply.  In addition, the only members of
 this set with no odd prime factors at most 17 are $n+2$ and $n+10$.  So, condition (2) does not apply either.
However, for $r\ge17$, we can strengthen (2) to
 \begin{itemize}
     \item[(\rm 2')] There exist $i,j \in \{1,\ldots,r\}$ such that $0<|i-j|<16$ and both $n+i$ and $n+j$ are even but not divisible by any odd prime up to $r$.
 \end{itemize}
 In addition, we have found it easier at higher levels to use Proposition \ref{prop:suff} directly for
 those intervals that do not have a member coprime to $m_r$.  When $r=17$, there are 498 such intervals $[n,n+16]$ 
 below 255255.  For each such interval $I$, we examine $I$ translated by $j\times510510$ for $j=0,1,2,3$, searching
 in each for a member of the form $ab$ where the primes in $b$ are greater
 than $r$, and with $a=2$ or~4.  All but two intervals had this property, namely the length 17 intervals
 starting at $n=60462$ and at $n=97590$ shifted
 by $3\times 510510$.  However, in these intervals, property (2') applies. 
 
 In continuing on to $r=19,20,21,22$ it turns out that conditions (1) and (2') are not sufficient for all
 cases.  We can supplement with a new condition which works for $r\ge13$:
 \begin{itemize}
     \item[\rm(3)] There exist $i,j\in\{1,\dots,r\}$ such that $9\nmid j-i$, both $n+i$ and $n+j$ are multiples
     of 3, they are not divisible by any prime in $[5,r]$, and they are not divisible by 8.
 \end{itemize}
 The sufficiency of condition (3) follows from Poposition \ref{prop:suff} with $a\mid 12$, using that the
 largest prime factor of $2^{12}-1$ is 13.
Condition (3) works well in conjunction with condition (2') since to apply them one needs to to translate the interval
 by $jm_{r}$, for $j=0,1,2,3$.  In examining length 19 intervals,
 all but 8439 of them satisfy condition (1).  Looking up to 4 times $m_{19}$,
 all but a handful of these $4\times8439$ intervals satisfy the hypothesis of Proposition \ref{prop:suff} with $a=2$ or 4.
 This handful is settled using conditions (2') and (3).
 This completes the proof.
 \end{proof}
 \begin{remark}
One possible route to proving that $S_r(n)$ is always nonintegral is to show that one of (1), (2) in
Theorem \ref{thm:tt} always occurs.  In fact, in the next section we show that condition (1) holds
when $r$ is large for most of the intervals $\{n+1,\dots,n+r\}$.  However, there are exceptional intervals where
(1) does not hold, and we have already seen that there can be intervals where neither (1) nor (2) hold.  In addition,
one can show that for all sufficiently large numbers $r$
there is some $n$ such that $\{n+1,\dots,n+r\}$ has each member with an odd prime factor at most $r$.
For example, a short argument shows this is the case for $r=103$.  So, replacing condition (2') with
higher powers of 2 does not always work either.    It
is conceivable that for every $r$ and every interval of $r$ consecutive integers at least $r$ there is a member for which
the hypothesis of Proposition \ref{prop:suff} holds, but we are not sure if this is so.  Complicating things,
one has for all sufficiently large numbers $r$ an interval of $r$ consecutive integers each divisible by
a prime $p$ in the range $\log_2r<p\le r$, see \cite[equation (3)]{P}.
\end{remark}

\section{Density}\label{sec:main}
In order to prove Theorem \ref{thm:dense}, we first show that condition (1) from Theorem \ref{thm:tt} usually holds.
\begin{theorem}\label{thm:main}
For each $k>0$ there exists a constant $c_{k}>0$, such that for each integer $r>1$, the asymptotic density of those integers $n$ such that $\{n+1,\ldots,n+r\}$ contains no number coprime to $m_r$ is at most $c_{k}/r^{k}$.
\end{theorem}
\begin{proof} Let $\alpha\ge1$ be a real number to be determined.  Recall that
\begin{center}
$V_{\alpha}(q)=\displaystyle \sum_{i=1}^{\phi(q)} (a_{i+1}-a_{i})^{\alpha}$
\end{center}
where $1=a_{1}<a_{2} \cdots < a_{\phi(q)+1}$ are the integers in $[1,q+1]$ that are coprime to $q$. By Proposition \ref{mont}, it follows that 
\begin{center}
$V_{\alpha}(q)<c(\alpha)\phi(q)(\phi(q)/q)^{-\alpha}=c(\alpha)q(q/\phi(q))^{\alpha-1}$ 
\end{center}
for some constant $c(\alpha)>0$. Applying this with $q=m_r$, together with \eqref{eq:euler}, yields
\begin{equation} \label{ineq1}
V_{\alpha}(q)<c(\alpha)q(3\operatorname{ln}r)^{\alpha-1}.
\end{equation}
Let $N:=\displaystyle \sum_{\left\{i: \ a_{i+1}-a_{i} \geq r\right\}} (a_{i+1}-a_{i})$. Then
\begin{align*}
Nr^{\alpha-1}=\displaystyle \sum_{\left\{i: \ a_{i+1}-a_{i} \geq r\right\}} (a_{i+1}-a_{i}) r^{\alpha-1} 
\leq \displaystyle \sum_{i=1}^{\phi(q)} (a_{i+1}-a_{i})^{\alpha} 
=V_{\alpha}(q),
\end{align*}
so that $N \leq V_{\alpha}(q)/r^{\alpha-1}$. Using inequality (\ref{ineq1}) we obtain 
\begin{equation} \label{ineq2}
N \leq c(\alpha) q \left(\frac{3\operatorname{ln}r}{r}\right)^{\alpha-1}.
\end{equation}

Now, we note that if an interval $I\subset[1,m_r]$ of integers does not contain any member of $\{a_1,\dots,a_{\phi(q)+1}\}$, then since it is an interval, it must lie completely between two consecutive members of this set. 
So, if $\{n+1,\ldots,n+r\}$ contains no number coprime to $m_r=q$, there exists $w \in [1,\phi(q)]$ such that $(n,n+r+1) \subseteq (a_{w},a_{w+1})$. Then $n$ may be any of the numbers $a_w,a_w+1,\dots,a_{w+1}-r-1$,
that is the interval $(a_w,a_{w+1})$ gives rise to exactly $a_{w+1}-a_w-r$ intervals $\{n+1,\dots,n+r\}$.
Therefore
\begin{center}
$\#\{1 \leq n \leq q:  \gcd(n+i,q) \neq 1 \hbox{ for} \ i=1,\dots,r\} \leq \displaystyle \sum_{\{i: a_{i+1}-a_{i}>r\}} (a_{i+1}-a_{i}-r) \leq N$.
\end{center}

It remains to note that the integers $n$ where $\{n+1,\ldots,n+r\}$ has no element coprime to $m_r$ form a periodic set mod $m_r$.  That is, if $n$ has this property, so does every positive integer $m\equiv n\pmod{r!}$.
Hence by (\ref{ineq2}) the density of the set of such numbers is at most $c(\alpha)\left(\frac{3\operatorname{ln}r}{r}\right)^{\alpha-1}$. Then let $\alpha=k+2$, and the result follows
with $c_k$ the maximal value of $c(\alpha)(3\ln r)^{\alpha-1}/r$. 
\end{proof}

Theorem \ref{thm:dense} now follows as a corollary.
\begin{proof}[Proof of Theorem \ref{thm:dense}]
It follows from Theorem \ref{thm:tt} that if some member of $\{n+1,\dots,n+r\}$ is
coprime to $m_r$, then $S_r(n)$ is nonintegral.  Thus, the theorem follows immediately from Theorem \ref{thm:main}.
\end{proof}

\section{Thongjunthug's theorem}
\label{sec:tt}

In \cite{10}, the fourth author proved the following theorem.
\begin{theorem}
\label{thm:ttorig}
{\rm (\cite[Theorem 4.1]{10})} Given an integer $r \geq 5$, the sum $S_{r}(n)$ is not a positive integer for all $n \geq 2$ if the following two conditions hold:
\begin{enumerate}[\rm(a)]
\item For all $l \in \{1,2,\ldots, r\}$, we have
\begin{center}
$r(F_{r}(n)-(-1)^{r-1}(r-1)!) \equiv \pm r!(2^{n+l}-1) \pmod{(n+l)}$
\end{center}
where $F_{r}(n)=\displaystyle \sum_{i=0}^{r-1} \displaystyle \sum_{k=0}^{n+r} s(r,i+1) k^{i} \binom{n+r}{k} $, and $s(r,i+1)$ is the signed Stirling number of the first kind. 
\item Each integer $n \in \{0,1,\ldots,P-1\}$, where $P$ is the product of all primes up to $r$, satisfies at least one of the following: 
\begin{enumerate}
\item[\rm(b1)] There exists $i \in \{1, \ldots, r\}$ such that $p \nmid (n+i)$ for all primes $p \leq r$. 
\item[\rm(b2)] There exist $i,j \in \{1,\ldots,r\}$ such that $0<|i-j|<8$ and both $n+i$ and $n+j$ are even but not divisible by any odd prime up to $r$.
\end{enumerate}
\end{enumerate}
\end{theorem}
We have seen in Theorem \ref{thm:tt} that this result holds without condition (a), so that condition is
superfluous.  However, condition (a) is harmless, in that it always holds.  We now prove this assertion.
\begin{proposition}\label{prop:tt}
The condition (a) in Theorem \ref{thm:ttorig} holds for all $r,n\in\mathbb{N}$.
\end{proposition}
\begin{proof}
Using \cite[Lemma 3.2]{10} and \eqref{eq:2pow} one has 
\begin{align*}
r(F_{r}(n)-(-1)^{r-1}(r-1)!)&=(n+1)\cdots(n+r)S(r,n)\\
&=\sum_{i=1}^{r} (-1)^{r-i} r \binom{r-1}{i-1} (2^{n+i}-1)\prod_{j=1,j\ne i}^r(n+j).
\end{align*}
Thus, for each $i=1,\dots,r$ we have
\begin{equation}
    \label{eq:ident}
r(F_{r}(n)-(-1)^{r-1}(r-1)!)\equiv (-1)^{r-i} r \binom{r-1}{i-1} (2^{n+i}-1)\prod_{j=1,j\ne i}^r(n+j)\pmod{n+i}.
\end{equation}
Since $n+j\equiv j-i\pmod{n+i}$, we have
\[
\prod_{j=1,j\ne i}^r(n+j)\equiv \prod_{j=1,j\ne i}^r(j-i)=\pm (r-i)!(i-1)!\pmod{n+i}.
\]
Multiplying this by $r\binom{r-1}{i-1}$ one gets $\pm r!$, and substituting into \eqref{eq:ident} gives
\[
r(F_{r}(n)-(-1)^{r-1}(r-1)!)\equiv \pm r!(2^{n+i}-1)\pmod{n+i},
\]
which completes the proof.
\end{proof}

\end{document}